\documentclass[portuges,12pt,letter]{article}
\usepackage[centertags]{amsmath}
\usepackage{amsfonts}
\usepackage{newlfont}
\usepackage{amscd}
\usepackage{graphics}
\usepackage{epsfig}
\usepackage{indentfirst}
\usepackage{amsxtra}
\usepackage[latin1]{inputenc}
\usepackage{amssymb, amsmath}
\usepackage{amsthm}
\usepackage[mathscr]{eucal}

\newtheorem{thm}{Theorem}[section]

\newtheorem{remark}[thm]{Remark}

\author{Fabio Silva Botelho \\ Department of Mathematics \\  Federal University of Santa Catarina, UFSC \\
Florian\'{o}polis, SC - Brazil}

\title{\bf  Duality suitable for a class of non-convex optimization problems} 

\begin{document}
\maketitle

\abstract{ In this article we develop a duality principle suitable for a large class of problems in optimization.
The main result is obtained through basic tools of convex analysis and duality theory. We establish a correct relation between the critical
points of the primal and dual formulations and formally prove there is no duality gap between such formulations, in a local extremal
context. }

\section{Introduction} This short letter develops duality for a class of problems in $\mathbb{R}^n$.
We consider the problem of minimizing the functional $J:\mathbb{R}^n \rightarrow \mathbb{R}$ where
$$J(x)=\frac{1}{2} x^T A x +\sum_{j=1}^N \frac{\gamma_j}{2}\left(\frac{x^TB_jx}{2}+c_j\right)^2-f^T x, \; \forall x \in \mathbb{R}^n$$
where
$A$ is a symmetric $n \times n$ matrix, $B_j$ is a symmetric $n \times n$ matrix,  $\;\forall j \in \{1,\cdots,N\}$ and $c_j,\;\gamma_j \in \mathbb{R}$,
$\gamma_j>0,$ $\forall j \in \{1, \cdots,N\}.$ Moreover $f \in \mathbb{R}^n$ is a fixed vector.

In this case we do not assume $n=N$ and the results are valid even for the case $n \neq N$, $\;\forall n,N \in \mathbb{N}.$

\begin{remark} About the notation for a generic $n \times n$ real matrix  $A$ we denote $A>\mathbf{0}$ if $$x^T Ax>0,\; \forall x \in \mathbb{R}^n, \text{ such that } x \neq \mathbf{0}.$$

Similarly, we denote $A>B$, if $A-B>\mathbf{0}.$ Moreover $x^T$ and $A^T$ denotes the transpose of a vector in $\mathbb{R}^n$ and for a $n \times n$  matrix, respectively. Finally, $I_d$ denotes the $n\times n$ identity matrix.
\end{remark}.

\begin{remark} About the references, we must emphasize our work is a kind of extension and continuation of the original works of Bielski and Telega \cite{2900,85} combined with the work of Toland \cite{12}. The technical details follow in some extent the results in \cite{12a}. Anyway, we highlight once more our work in some sense complements the  results
in \cite{2900,85} but now applied to a $\mathbb{R}^n$ simpler context.

Similar problems have been addressed in \cite{17,9}, among others.
\end{remark}

\section{The main duality principle}
Our main result is summarized by the following theorem.
\begin{thm} Consider the main functional $J: \mathbb{R}^n \rightarrow \mathbb{R}$ where
$$J(x)=\frac{1}{2} x^T A x +\sum_{j=1}^N \frac{\gamma_j}{2}\left(\frac{x^TB_jx}{2}+c_j\right)^2-f^T x, \; \forall x \in \mathbb{R}^n,$$
with the assumptions about matrices, vectors and real constants stated in the last section.

Define also $J_1:\mathbb{R}^n \times \mathbb{R}^N \rightarrow \mathbb{R}$ and $J_2:\mathbb{R}^n \times \mathbb{R}^n \times B^* \rightarrow \mathbb{R}$
by
\begin{eqnarray}
J_1(x,v_0^*)&=&\frac{1}{2} x^TAx+\sum_{j=1}^N (v_0^*)_j \left( \frac{x^T B_j x}{2}+c_j\right) \nonumber \\ &&
-\sum_{j=1}^N \frac{(v_0^*)_j^2}{2\gamma_j}-f^T x,
\end{eqnarray}
and
\begin{eqnarray}
J_2(x,v^*,v_0^*)&=& \frac{1}{2} x^TAx-f^Tx \nonumber \\ &&+ \frac{1}{2}(v^*)^T\left(-\sum_{j=1}^N (v_0^*)_j B_j+K I_d\right)^{-1}v^*
\nonumber \\ && -(v^*)^T x +\frac{K}{2}x^Tx \nonumber \\ &&-\sum_{j=1}^N \frac{(v_0^*)_j^2}{2\gamma_j}+\sum_{j=1}^N (v_0^*)_j c_j,
\end{eqnarray}
where $$B^*=\left\{v_0^* \in \mathbb{R}^N\::\; \sum_{j=1}^N (v_0^*)_jB_j+KI_d> \frac{K I_d}{2}\right\}.$$

Assume $x_0 \in \mathbb{R}^n$ is such that
$\delta J(x_0)=\mathbf{0}$ and define
$$(\hat{v}_0^*)_j=\gamma_j\left(\frac{1}{2}x_0^T B_j x_0 +c_j\right),\; \forall j \in \{1, \cdots,N\}$$
and
$$\hat{v}^*=-\sum_{j=1}^N (v_0^*)_jB_jx_0+Kx_0.$$

Define also $J^*:\mathbb{R}^n \times \mathbb{R}^N \rightarrow \mathbb{R}$ by
\begin{eqnarray}
J^*(v^*,v_0^*)&=&- \frac{1}{2} (v^*+f)^T (K I_d+A)^{-1}(v^*+f) \nonumber \\ &&
+\frac{1}{2}(v^*)^T\left(-\sum_{j=1}^N (v_0^*)_j B_j+K I_d\right)^{-1}v^* \nonumber \\ &&-\sum_{j=1}^N \frac{(v_0^*)_j^2}{2\gamma_j}+\sum_{j=1}^N (v_0^*)_j c_j.
\end{eqnarray}
Under such hypotheses
$$\delta J^*(\hat{v}, \hat{v}_0^*)=\mathbf{0}$$ and
$$J(x_0)=J^*(\hat{v}^*,\hat{v}_0^*).$$

Moreover, for $K>0$ sufficiently big,
\begin{enumerate}
\item\label{1} if $\delta^2 J(x_0)> \mathbf{0}$, then there exist $r,r_1,r_2>0$ such that
\begin{eqnarray}
J(x_0)&=& \inf_{x \in B_r(x_0)} J(x)\nonumber \\ &=& \inf_{ v^* \in B_{r_1}(\hat{v}^*)}\left\{\sup_{v_0^* \in B_{r_2}(\hat{v}_0^*)}J^*(v^*,v_0^*)\right\}
\nonumber \\ &=& J^*(\hat{v}^*,\hat{v}_0^*).
\end{eqnarray}

\item\label{2} If $\hat{v}_0^* \in E^*=A_+^* \cap B^*$, where
$$A_+^*=\left\{v_0^* \in \mathbb{R}^N\;:\;\sum_{j=1}^N (v_0^*)_jB_j+A>\mathbf{0}\right\},$$
then there exists $r_2>0$ such that
\begin{eqnarray}
J(x_0)&=& \inf_{ x \in \mathbb{R}^n} J(x) \nonumber \\ &=&
\inf_{v^* \in \mathbb{R}^n} \left\{ \sup_{ v_0^* \in B_{r_2}(\hat{v}_0^*) \cap E^*} J^*(v^*,v_0^*)\right\} \nonumber \\ &=&
J^*(\hat{v}^*,\hat{v}_0^*).
\end{eqnarray}
\item\label{3} If $\delta^2J(x_0)< \mathbf{0}$ so that $\hat{v}_0^* \in A^*_-$, where
$$A^*_-=\left\{v_0^* \in \mathbb{R}^N\;:\; \sum_{j=1}^N (v_0^*)_jB_j+A <\mathbf{0}\right\},$$
then there exist $r,r_1,r_2>0$ such that
\begin{eqnarray}
J(x_0)&=& \sup_{x \in B_r(x_0)} J(x) \nonumber \\ &=&
\sup_{v^* \in B_{r_1}(\hat{v}^*)} \left\{ \sup_{ v_0^* \in B_{r_2}(\hat{v}_0^*)} J^*(v^*,v_0^*)\right\}
\nonumber \\ &=& J^*(\hat{v}^*,\hat{v}_0^*).
\end{eqnarray}
\end{enumerate}
\end{thm}
\begin{proof}
The proof that $\delta J^*(\hat{v}^*,\hat{v}_0^*)=\mathbf{0}$ and $J(x_0)=J^*(\hat{v}^*,\hat{v}_0^*)$ results directly from the Legendre
transform standard properties.

Now suppose $\delta^2J(x_0)> \mathbf{0}.$ For $K>0$ sufficiently big $J^*(v^*,v_0^*)$ is concave in $v_0^*$ in a neighborhood of $\hat{v}_0^*$,
so that from the min-max theorem we may obtain $r,r_1,r_2>0$ such that
\begin{eqnarray}
J(x)&=& \sup_{v_0^* \in B_{r_2}(\hat{v}_0^*)} \left\{ \inf_{v^* \in B_{r_1}(\hat{v}^*)} J_2(x,v^*,v_0^*)\right\} \nonumber \\ &=&
 \inf_{v^* \in B_{r_1}(\hat{v}^*)}\left\{ \sup_{v_0^* \in B_{r_2}(\hat{v}_0^*)} J_2(x,v^*,v_0^*)\right\},\; \forall x \in B_r(x_0).
 \end{eqnarray}
Hence for some not relabeled $r>0$, we have
\begin{eqnarray}
J^*(\hat{v}^*,\hat{v}_0^*) &=& J(x_0) \nonumber \\ &=& \inf_{x \in B_r(x_0)} \left\{ \inf_{ v^* \in B_{r_1}(\hat{v}^*)}\left\{
\sup_{v_0^* \in B_{r_2}(\hat{v}_0^*)} J_2(x,v^*,v_0^*)\right\}\right\}
\nonumber \\ &=&
 \inf_{ v^* \in B_{r_1}(\hat{v}^*)} \left\{\inf_{x \in B_r(x_0)}\left\{
\sup_{v_0^* \in B_{r_2}(\hat{v}_0^*)} J_2(x,v^*,v_0^*)\right\}\right\}
\nonumber \\ &=& \inf_{v^* \in B_{r_1}(\hat{v}^*)} \left\{\sup_{v_0^* \in B_{r_2}(\hat{v}_0^*)}\left\{ \inf_{x \in B_r(x_0)} J_2(x,v^*,v_0^*)\right\}\right\}
\nonumber \\ &=& \inf_{v^* \in B_{r_1}(\hat{v}^*)} \left\{\sup_{v_0^* \in B_{r_2}(\hat{v}_0^*)} J^*(v^*,v_0^*)\right\}.
\end{eqnarray}

The proof of item \ref{1} is complete.
 Suppose now $\hat{v}_0^* \in E^*$.

 First observe that $$\frac{\partial^2 J^*(v^*,v_0^*)}{\partial (v^*)^2}>\mathbf{0}, \; \forall v^* \in \mathbb{R}^n,\; v_0^* \in E^*$$

 From the min-max theorem, for $K>0$ sufficiently big, we may find $r_2>0$ such that
 $$J(x_0)=J^*(\hat{v}^*,\hat{v}_0^*)=\inf_{v^* \in \mathbb{R}^n} \left\{ \sup_{v_0^* \in B_{r_2}(\hat{v}_0^*) \cap E^*} J^*(v^*,v_0^*) \right\}.$$

 In particular,

 \begin{eqnarray}
 J(x_0) &\leq& J^*(v^*,\hat{v}_0^*)
 \nonumber \\ &\leq& \frac{1}{2}x^T A x -f^Tx\nonumber \\ &&+ \frac{1}{2}(v^*)^T\left(-\sum_{j=1}^N (\hat{v}_0^*)_j B_j+K I_d\right)^{-1}v^*-(v^*)^Tx
 \nonumber \\ &&+\frac{K}{2} x^Tx-\sum_{j=1}^N \frac{(\hat{v}_0^*)_j^2}{2\gamma_j}+\sum_{j=1}^N (\hat{v}_0^*)_j c_j,
 \end{eqnarray}
$ \forall x \in \mathbb{R}^n,   \; v^* \in \mathbb{R}^n.$

 From this we get
 \begin{eqnarray}
 J(x_0) &\leq& \sup_{v_0^*\in  B_{r_2}(\hat{v}_0^*)} \left\{ \inf_{v^* \in \mathbb{R}^n} \left\{
 \frac{1}{2}x^T A x-f^Tx \right.\right. \nonumber \\ &&+ \frac{1}{2}(v^*)^T\left(-\sum_{j=1}^N (v_0^*)_j B_j+K I_d\right)^{-1}v^*-(v^*)^Tx
 \nonumber \\ && \left.\left.+\frac{K}{2} x^Tx-\sum_{j=1}^N \frac{(v_0^*)_j^2}{2\gamma_j}+\sum_{j=1}^N (v_0^*)_j c_j\right\}\right\} \nonumber \\ &\leq&
  \sup_{v_0^* \in \mathbb{R}^N} \left\{ \inf_{v^* \in \mathbb{R}^n} \left\{
 \frac{1}{2}x^T A x-f^Tx \right.\right. \nonumber \\ &&+ \frac{1}{2}(v^*)^T\left(-\sum_{j=1}^N (v_0^*)_j B_j+K I_d\right)^{-1}v^*-(v^*)^Tx
 \nonumber \\ && \left.\left.+\frac{K}{2} x^Tx-\sum_{j=1}^N \frac{(v_0^*)_j^2}{2\gamma_j}+\sum_{j=1}^N (v_0^*)_j c_j\right\}\right\}
 \nonumber \\ &=& J(x), \; \forall x \in \mathbb{R}^n.
 \end{eqnarray}

 Summarizing these last results, we have obtained,
 \begin{eqnarray}
J(x_0)&=& \inf_{ x \in \mathbb{R}^n} J(x) \nonumber \\ &=&
\inf_{v^* \in \mathbb{R}^n} \left\{ \sup_{ v_0^* \in B_{r_2}(\hat{v}_0^*) \cap E^*} J^*(v^*,v_0^*)\right\} \nonumber \\ &=&
J^*(\hat{v}^*,\hat{v}_0^*).
\end{eqnarray}

Finally assume $\delta^2J(x_0)< \mathbf{0},$ so that
$\hat{v}_0^* \in A_-^*.$

Since $A_-^*$ is open, we may obtain $r_2>0$ such that $$\frac{\partial^2 J^*(v^*,v_0^*)}{\partial (v^*)^2}< \mathbf{0},\; \forall v^* \in \mathbb{R}^n,\;
 v_0^* \in B_{r_2}(\hat{v}_0^*).$$

 From such assumptions and results, we may obtain $r,r_1,r_2>0$ such that
 $$J(x_0)=\sup_{ x \in B_r(x_0)} J(x),$$
 and
 $$\sup_{x \in B_{r}(x_0)} J_1(x,v_0^*) = \sup_{v^* \in B_{r_1}(\hat{v}^*)} J^*(v^*,v_0^*),\; \forall v_0^* \in B_{r_2}(\hat{v}_0^*).$$

 From these results and assumptions, for a not relabeled $r>0$ we have
 \begin{eqnarray}
 J^*(\hat{v}^*,\hat{v}_0^*)&=& J(x_0)
 \nonumber \\ &=& \sup_{x \in B_r(x_0)} J(x) \nonumber \\ &=&
 \sup_{x \in B_r(x_0)}\left\{ \sup_{v_0^* \in B_{r_2}(\hat{v}_0^*)} J_1(x,v_0^*)\right\}
 \nonumber \\ &=&
  \sup_{v_0^* \in B_{r_2}(\hat{v}_0^*)} \left\{\sup_{x \in B_r(x_0)}J_1(x,v_0^*)\right\}
   \nonumber \\ &=&
  \sup_{v_0^* \in B_{r_2}(\hat{v}_0^*)} \left\{\sup_{v^* \in B_{r_1}(\hat{v}^*)}J^*(v^*,v_0^*)\right\}.
  \end{eqnarray}
 Summarizing these last results, we have obtained
 \begin{eqnarray}
J(x_0)&=& \sup_{x \in B_r(x_0)} J(x) \nonumber \\ &=&
\sup_{v^* \in B_{r_1}(\hat{v}^*)} \left\{ \sup_{ v_0^* \in B_{r_2}(\hat{v}_0^*)} J^*(v^*,v_0^*)\right\}
\nonumber \\ &=& J^*(\hat{v}^*,\hat{v}_0^*).
\end{eqnarray}

The proof is complete.
\end{proof}

\section{Conclusion} In this article we have developed a duality principle suitable for a large class of optimization problems in
$\mathbb{R}^n$. We highlight the min-max theorem has a fundamental role for the proofs of the main results.

We believe these results may be extended to more complex  variational models such as non-linear models of plates and shells.


\begin{thebibliography}{}
%
%



\bibitem{2900} W.R. Bielski, A. Galka, J.J. Telega, The Complementary Energy Principle and Duality for
Geometrically Nonlinear Elastic Shells. I. Simple case of moderate rotations around a tangent to the middle surface.
Bulletin of the Polish Academy of Sciences, Technical Sciences, Vol. 38, No. 7-9, 1988.
\bibitem{85} W.R. Bielski and J.J. Telega, A Contribution to Contact Problems for a Class of Solids and Structures,
Arch. Mech., 37, 4-5, pp. 303-320, Warszawa 1985.

 \bibitem{12a}
F. Botelho, {Functional Analysis and Applied Optimization in Banach Spaces},
 (Springer Switzerland, 2014).
 \bibitem{19} F. Botelho, {Real Analysis and Applications}, (Springer Switzerland, 2018).
 \bibitem{17} D.Y. Gao and H.F. Yu,  Multi-scale modelling and canonical dual finite element method in phase transition in solids.
Int. J. Solids Struct., 45, 3660-3673 (2008).


 \bibitem{9} D.Y.Gao and C. Wu, On the Triality Theory in Global Optimization, Arxiv: 1104.2970 - v2, February, 2012. 


\bibitem{12} J.F. Toland, {\it A duality principle for non-convex
optimisation and the calculus of variations}, Arch. Rath. Mech.
Anal., {\bf 71}, No. 1 (1979), 41-61.



\end{thebibliography}
\end{document}